\documentclass[reqno, 11pt,a4letter]{amsart}
\usepackage{amsmath,amsxtra,amssymb,latexsym, amscd,amsthm}
\usepackage[utf8]{inputenc}
\usepackage[mathscr]{euscript}
\usepackage[english]{babel}
\usepackage[active]{srcltx}
\usepackage{color}

\newtheorem {theorem}{Theorem}[section]
\newtheorem {claim}{Claim}[section]
\newtheorem {corollary}{Corollary}[section]
\newtheorem {lemma}{Lemma}[section]

\newtheorem {definition}{Definition}[section]
\newtheorem {remark}{Remark}[section]
\newtheorem {proposition}{Proposition}[section]

\newcommand{\dist}{{\rm dist}}
\newcommand{\reg}{{\rm reg}}
\newcommand{\sing}{{\rm sing}}
\newcommand{\graph}{{\rm graph}}
\newcommand{\grad}{{\rm grad}}

\newcommand{\rank}{{\rm rank}}
\renewcommand{\span}{{\rm span}}
\renewcommand{\Re}{{\rm Re}}
\renewcommand{\Im}{{\rm Im}}
\newcommand{\R}{\Bbb R}
\newcommand{\B}{\Bbb B}
\newcommand{\C}{\Bbb C}

\newcommand{\N}{\Bbb N}
\renewcommand{\P}{\Bbb P}

\newcommand{\Z}{\mathcal Z}
\renewcommand{\S}{\mathcal S}

\title{Thom isotopy theorem for non proper maps and computation of sets of stratified generalized critical values}

\author{S\~i Ti\d{\^e}p \DH inh$^\dagger$}
\address{Institute of Mathematics, Polish Academy of Sciences, \'Sniadeckich 8, 00-656 Warsaw, Poland and
Institute of Mathematics, VAST, 18, Hoang Quoc Viet Road, Cau Giay District 10307, Hanoi, Vietnam}
\email{sdinh@impan.pl or dstiep@math.ac.vn}

\author{Zbigniew Jelonek$^\ddagger$}
\address{Institute of Mathematics, Polish Academy of Sciences, \'Sniadeckich 8, 00-656 Warsaw, Poland}
\email{najelone@cyf-kr.edu.pl}

\thanks{$^\dagger$This author is partially supported by Vietnam National Foundation for Science and Technology Development (NAFOSTED) grant 101.04-2017.12, $^\ddagger$This  author is partially supported by the grant of Narodowe Centrum Nauki, grant number 2015/17/B/ST1/02637}



\date{ \today}

\begin{document}
\maketitle

\begin{abstract} Let $X\subset\C^n$ be an affine variety and $f:X\to\C^m$ be the restriction to $X$ of a polynomial map $\C^n\to\C^m$. In this paper, we construct an affine Whitney stratification of $X$. The set $K(f)$ of stratified generalized critical values  of $f$ can be also computed. We show that $K(f)$ is a nowhere dense  subset of $\C^m$, which contains the set $B(f)$ of bifurcation values  of $f$ by proving a version of the isotopy lemma for non-proper polynomial maps on singular varieties.
\end{abstract}

\pagestyle{plain}

\section{Introduction} 

Ehresmann's fibration theorem \cite{Ehresmann1950} states that a proper smooth surjective submersion $f:X\to N$ between smooth manifolds is a locally trivial fibration. With some extra assumptions, this result has been considered in different contexts. 

Firstly, if we remove the assumption of properness or smoothness, in general, Ehresmann's fibration theorem does not hold since $f$ might have ``local singularities" or ``singularities at infinity". The set of points in $N$ where $f$ fails to be trivial, denoted by $B(f)$, is called the {\bf bifurcation set} of $f$, which is the union of the set $K_0(f)$  of {\bf critical values} and the set $B_\infty(f)$ of {\bf bifurcation values at infinity} of $f$. So far, characterizing $B_\infty(f)$ is still an open problem. In general, people use a bigger set (but easier to describe),  the {\bf set of asymptotic critical values of $f$}, denoted by $K_\infty(f)$, to control $B_\infty(f)$. The set $K_\infty(f)$ is always a nowhere dense subset of $\C^m$ and it is a good aproximation of the set $B_\infty(f).$ For dominant maps on smooth complex affine varieties, the computation of $K_\infty(f)$, and hence of the {\bf set of generalized critical values}, $K(f):=K_0(f)\cup K_\infty(f)$, is given in \cite{Jelonek2005}.

Now if $X$ is singular, we need to partition $X$ into disjoint smooth manifolds and then apply Ehresmann's fibration theorem on each part. However, if we do not require any extra assumption, then the trivialization on the parts may not match. This obstacle can be overcome by introducing the Whitney conditions \cite{Whitney1965-1,Whitney1965-2}. Indeed, if $f$ is proper and $X$ admits a Whitney stratification, then $f$ is locally trivial if it is a submersions on stratas \cite{Thom1969,Mather2012,Verdier1976}. Moreover, if $f$ is non proper and non smooth, we can also define the bifurcation set of $f$ such that $f$ is locally trivial outside $B(f)$. However, so far, to our knowledge, no connection between $B(f)$ and the \textbf{set of stratified generalized critical values of $f$}, defined by $K(f):=\bigcup_{X_\alpha\in\S}K(f|_{X_\alpha})$, for a  Whitney stratification  $\S$ of $X$, has been established. 

Let $X\subset \C^n$ be a singular algebraic set of dimension $n-r$ with $I(X) = \{g_1,\dots,g_{\omega}\}$ and let $f := (f_1,\dots,f_m) : X \to \C^m$ be a polynomial dominant map. Now restricting ourselves to the cases of dominant polynomial maps on singular affine varieties, the main goals of this paper are the following:

\begin{itemize}
\item Construct an affine Whitney stratification $\S$ of $X.$
\item Establish some version of the Thom isotopy lemma for $f$ which yield the inclusion $B(f)\subset \bigcup_{X_\alpha\in\S}K(f|_{X_\alpha})$.
\item Calculate the set of stratified generalized critical values of $f$ given by $K(f):=\bigcup_{X_\alpha\in\S}K(f|_{X_\alpha})$.
\end{itemize}

The paper is organized as follows. In Section \ref{Whitney}, we recall the definitions of Whitney regularity and Whitney stratification, then we construct an affine stratification from a filtration of $X$ by means of some hypersurfaces, and refine it to get an affine Whitney stratification. Some versions of the Thom isotopy lemma for non-proper polynomial maps (Theorem \ref{NonProperIsotopy} and Corollary \ref{NonProperIsotopySubmersion}) will be given in Section \ref{Isotopy}. Then we compute the set of stratified generalized critical values of $f$, which contains the bifurcation values of $f$, where $f := (f_1,\dots,f_m) : X \to \C^m$ is a polynomial dominant map, in the last Sections \ref{StratifiedKf} and \ref{B_i}.

For the remainder of the paper, the differential of $f$ at a point $x$ is identified with its (row) matrix, so we write $d_xf=\Big(\frac{\partial f}{\partial x_1}(x),\dots,\frac{\partial f}{\partial x_n}(x)\Big)$. Let $$\nabla f(x):=\begin{bmatrix} 
\overline{\frac{\partial f}{\partial x_1}(x)}\\
\vdots\\
\overline{\frac{\partial f}{\partial x_n}(x)}
\end{bmatrix},$$
the Hermitian transpose of $d_xf.$ For $v,w\in \C^n,$ denote by $\langle v,w\rangle=\sum_{i=1}^n\overline v_iw_i$ the Hermitian product, and let $v\cdot w=\sum_{i=1}^n v_iw_i$. For the set $A\subset\C^n$, set $\overline A:=\{x:\overline x\in A\}$ and let $\overline A^\Z$ be the Zariski closure of $A.$ For an algebraic variety $X$, the singular part and the regular part of $X$ are denoted respectively by $\sing(X)$ and $\reg(X)$. 


\section{Affine Whitney stratifications}\label{Whitney}


\subsection{Preliminaries}

For any two different points $x,y\in\C^n$, define the secant $\overline {xy}$ to be the line passing through the origin which is parallel to the line through $x$ and $y$.

A \textbf{stratification} $\S$ of $X$ is a decomposition of $X$ into a locally finite disjoint union $X = \displaystyle\bigsqcup_{\alpha\in I} X_\alpha$ of non-empty, non-singular, connected, locally closed subvarieties, called strata, such that the boundary $\partial X_\alpha$ of any stratum $X_\alpha$ is a union of strata. If, in addition, for each $\alpha$, the closure $\overline X_\alpha$ and the boundary $\partial X_\alpha:=\overline X_\alpha\setminus X_\alpha$ are affine varieties, then we call $\S$ an \textbf{affine stratification}. It is obvious that any affine stratification is finite.

For  linear subspaces $F,G\subseteq\C^n$, let 
$$\delta(F,G):=\sup_{\substack{x\in F\\ \|x\|=1}}\dist(x,G),$$
where $\dist(x,G)$ is the Hermitian distance. 

Let $(X_\alpha,X_\beta)$ be a pair of strata of $\S$ such that $X_\beta\subset\overline X_\alpha$ and let $x\in X_\beta.$ We recall  some regularity conditions:
\begin{itemize}
\item [(a)] The pair $(X_\alpha,X_\beta)$ is said to be \textbf{(a) Whitney regular at $x\in X_\beta$} if it satisfies the following Whitney condition (a) at $x$: if $x^k\in X_\alpha$ is any sequence such that $x^k\to x$ and $T_{x^k}X_\alpha\to T$, then $T\supset T_{x}X_\beta$. 

\item [(w)] The pair $(X_\alpha,X_\beta)$ is said to be \textbf{(w) regular} at $x\in X_\beta$ (or \textbf{(a) strict Whitney regular at $x$ with exponent $1$}) if it satisfies the following condition (w) at $x$: there exist a neighborhood $U$ of $x$ in $\C^n$ and a constant $c>0$ such that for any $y\in X_\alpha\cap U$ and $x'\in X_\beta\cap U$, we have $\delta(T_{x'}X_\beta,T_yX_\alpha)\leqslant c\|y-x'\|$.

\item [(b)] The pair $(X_\alpha,X_\beta)$ is said to be \textbf{Whitney regular at $x\in X_\beta$} if it satisfies the following Whitney condition (b) at $x$: for any sequence $x^k\in X_\alpha$ and $y^k\in X_\beta,\ y^k\not=x^k,$ such that $x^k\to x,\ y^k\to x,\ T_{x^k}X_\alpha\to T$ and $\overline{x^ky^k}$ converges to a line $\ell$ in the projective space $\P^{n-1}$, we have $\ell \subset T.$
\end{itemize}

We say that the pair $(X_\alpha,X_\beta)$ is \textbf{(a) Whitney regular} (resp. \textbf{Whitney regular}) if it is (a) Whitney regular (resp. Whitney regular) at every point of $X_\beta$.  We say that $\S$ is an \textbf{(a) Whitney stratification} (resp. a \textbf{Whitney stratification}) if any pair of strata $(X_\alpha,X_\beta)$ of $\S$ with $X_\beta\subset\overline X_\alpha$ is (a) Whitney regular (resp. {Whitney regular}). It is well-known that Whitney regularity implies (a) Whitney regularity \cite{Whitney1965-1,Whitney1965-2}. Moreover, in light of \cite{Teissier1982}, the Whitney condition (b) is equivalent to the condition (w) for the category of complex analytic sets, so to check the Whitney regularity, we can verify either the condition (w) or the condition (b).

For the purpose of this paper, we also need the following notion of Whitney (resp. (a) Whitney) regularity along a stratum. Let $X_\beta$ be a stratum of $\S$ and let $x\in X_\beta$. We say that $X_\beta$ is \textbf{Whitney regular} (resp. \textbf{(a) Whitney regular}) at $x$ if for any stratum $X_\alpha$ such that $X_\beta\subset\overline X_\alpha$, the pair $(X_\alpha,X_\beta)$ is Whitney (resp. (a) Whitney) regular at $x$. The stratum $X_\beta$ is \textbf{Whitney regular} (resp. \textbf{(a) Whitney regular}) if it is Whitney (resp. (a) Whitney) regular at every point of $X_\beta$. It is clear that $\S$ is a Whitney (resp. an (a) Whitney) stratification if and only if each stratum of $\S$ is Whitney (resp. (a) Whitney) regular.


\subsection{Construction of affine stratifications}\label{SectionAffine}

Let us, first of all, fix an affine stratification of $X$ whose construction is based on the following proposition.

\begin{proposition}\label{Affine}
Let $X\subset \C^n$ be an affine subvariety of pure codimension $r.$ Assume that $I(X)=\{g_1,\dots,g_\omega\}$, where $\deg g_i\le D.$ Let $W$ be an affine subvariety of positive codimension in $X$ with $I(W)=\{g_1,\dots,g_\omega, u_1,\dots,u_\tau\}$ where $u_i\not\in I(X)$ and $\deg u_i\leqslant D'$. 
Then there exist a polynomial $p_{X,W}$ on $\C^n$ of degree less than  or equal to $r(D-1)+D'$ such that $W\subseteq V(p_{X,W}):=\{x\in\C^n:\ p_{X,W}(x)=0\}$ and $X\setminus V(p_{X,W})$ is a smooth, dense subset of $X.$ Moreover, the polynomial $p_{X,W}$ can be constructed effectively.
\end{proposition}

\begin{proof}
Let $X=\bigcup^m_{i=1} X_i$, where $X_i$ are irreducible (hence $r$-codimensional) components of $X.$ Take  sufficiently generic (random) 
numbers $\alpha_{ij}\in \C$, $i=1,\dots,r$, $j=1,\dots,\omega$ and set
$$G_i=\sum_{j=1}^\omega \alpha_{ij}g_j,\ i=1,\dots,r.$$ 
Note that the set $Z:=V(G_1,\dots,G_r)$ has pure codimension $r$ and $X\subset Z.$ Let $\gamma_1,\dots,\gamma_\tau$ be some (random) generic numbers and set 
$$H:=\begin{cases}1 \ \ \ \ \ \ \ \ \ \ \ \text{if }\ W=\emptyset,\\
\sum_{i=1}^\tau \gamma_iu_{i} \ \text{otherwise}.
\end{cases}$$
Clearly $\dim \big(X\cap V(H)\big)<\dim X.$ Moreover, for a sufficiently general linear $r$-dimensional subspace $L^{r}\subset \C^n$ the intersection $L^r\cap Z$ has only isolated smooth points and $L^r\cap X_i \not=\emptyset$ for every $i=1,\dots,m.$ We can assume that $L^r$ is determined by the linear forms $l_i=\sum_{j=1}^n \beta_{ij} x_j,$ $i=1,\dots,n-r$, where  $\beta_{ij}$ sufficiently generic (random) 
numbers. Now take 
$$p_{X,W}=|Jac(G_1,\dots,G_r, l_1,\dots,l_{n-r})|\cdot H,$$
where $Jac(.)$ denotes the Jacobian matrix. Then $p_{X,W}$ is a polynomial with the required properties.

The polynomial $p_{X,W}$ can be find by using a probabilistic algorithm. First recall the following.

\begin{theorem}[\cite{grebner}] Let $I$
be an ideal in $k[x_1,\dots,x_n]$ and let $G=\{g_1,\dots,g_s\}$ be a Gröbner basis for $I$ with respect to a graded monomial order in 
$k[x_1,\dots,x_n].$ Then $G^h=\{g^h_1,\dots,g^h_s\}$ is a basis for $I^h\subset k[x_0,x_1,\dots,x_n].$
\end{theorem}

This theorem allows us to compute the set of points at infinity of an affine variety given by the ideal $I$, to this aim it is enough to compute the Groebner basis $\{g_1,\dots,g_s\}$ of the ideal $I$ and then to consider the ideal $I_\infty=\{x_0, g^h_1,\dots,g^h_s\}.$ Now we sketch the algorithm to compute the polynomial $p_{X,W}$. Note that for a given ideal $I$ we can compute dim $V(I)$ by \cite{dimension}.

\medskip

INPUT: The ideal $I=I(X)=\{g_1,\dots,g_\omega\}$ and the ideal $J=I(W)=\{g_1,\dots,g_\omega,u_1,\dots,u_\tau\}$

1) {\bf repeat }

choose random numbers $\alpha_{i1},\dots,\alpha_{i\omega}$, $i=1,\dots,r$;

put $G_i:=\sum^\omega_{k=1} \alpha_{ik} g_k$, $i=1,\dots,r$;

put $I=\{ G_1,\dots, G_r\}$;

compute the ideal $I_\infty=\{H_1,\dots,H_m\}\subset k[x_0,\dots,x_n]$

{\bf until}  $\dim V(I_\infty)=n-r$.

2) {\bf repeat }

choose random numbers $\beta_{i1},\dots,\beta_{in}$, $i=1,\dots,n-r$;

put $l_i:=\sum^n_{k=1} \beta_{ik} x_k$, $i=1,\dots,n-r$;

put $I=\{ G_1,\dots, G_r,l_1,\dots,l_{n-r}\}$;

compute the ideal $I_\infty=\{H_1,\dots,H_m\}\subset k[x_0,\dots,x_n]$;

{\bf if} dim $V(I_\infty)=0$ {\bf then}

{\bf begin}

compute $V(G_1,\dots,G_r, l_1,\dots,l_r):=\{a_1,\dots,a_p\}$

{\bf end} 

{\bf until}  $\dim V(I_\infty)=0$ and $|Jac(G_1,\dots,G_r,l_1,\dots,l_{n-r})(a_i)|\not=0$ for $i=1,\dots p.$

3)  {\bf repeat }

choose random numbers $\gamma_{1},\dots,\gamma_{\tau}$;

put $H:=\sum^\tau_{k=1} \gamma_{i} u_k$ ;

put $J=\{ G_1,\dots, G_r, H\}$;

compute the ideal $J_\infty\subset k[x_0,\dots,x_n]$

{\bf until}  $\dim V(J_\infty)<n-r$.

\vspace{3mm}

OUTPUT:  $p_{X,W}=|Jac(G_1,\dots,G_r,l_1,\dots,l_{n-r})|\cdot H$
\end{proof}

\begin{remark}\label{uwaga}
{\rm Let us assume that $I(X)$ and $I(W)$ are generated by polynomials from the ring $\Bbb F[x_1,\dots,x_n]$, where $\Bbb F$ is a subfield of $\C.$
Then we can choose a polynomial $p_{X,W}$ in this way that $p_{X,W}\in \Bbb F[x_1,\dots,x_n].$}
\end{remark}

Thus with no loss of generality, we can assume that $\rank Jac(g_1,\dots,g_r) = r$ on some non-empty regular open subset $X^0$ of $X$ and that $X=\overline{ X^0}$. 
It is clear that $V(p_{X,W})$ contains $\sing(X)\cup W$ and the singular points  of the projection  $(l_1,\dots,l_{n-r}): X\to \C^{n-r}.$   
Now to construct an affine stratification of $X$, it is enough to construct an affine filtration 
$X=X_0\supset X_1\supset\dots\supset X_{n-r}\supset X_{n-r+1}=\emptyset$ by induction with $X_{i+1}:=X_i\cap V(p_{X_i,\emptyset}),\ i=0,\dots,n-r$. The degree of each $X_i$ can be calculated and depends only on $D$.


\subsection{Construction of affine Whitney stratifications}\label{SectionWhitneyb}

In this section, we will construct an affine Whitney stratification of a given affine variety $X$, with $I(X)=\{g_1,\dots,g_\omega\}$ and $\deg g_i\le D$, by refining the affine stratification given in Subsection \ref{SectionAffine} so that the resulting stratification is still affine and moreover satisfies the Whitney condition.

First of all, inspired by the construction in \cite{Flores2016,Teissier1982}, let us describe the Whitney condition (b) algebraically. Assume that $Y\subset X$ is an affine subvariety of $X$ of dimension $n-p$ with $\dim Y<\dim X$ defined by 
$$Y:=X\cap\{\widetilde g_{r+1}=\dots=\widetilde g_{p}=0\}.$$
Set
$$\Gamma_1:=\left\{\begin{array}{lcll}
(x,y,w,v,\gamma,\lambda)\in\C^n\times\C^n\times\C^n\times\C^n\times\C\times \C^r: \\ 
 g_1(x)=\dots=g_r(x)=0\\
 g_1(y)=\dots=g_r(y)=\widetilde g_{r+1}(y)=\dots=\widetilde g_{p}(y)=0\\
 w=\gamma(x-y)\\
 v=\sum_{i=1}^r\lambda_id_x g_i\\
\end{array}\right\},$$ 
and let 
$$\pi_1:\C^n\times\C^n\times\C^n\times\C^n\times\C\times \C^r\to \C^n\times\C^n\times\C^n\times\C^n$$ be the projection on the first $4n$ coordinates. Let $C(X,Y)=\overline{\pi_1(\Gamma_1)}^\Z\subset (X\times Y\times\C^n\times\C^n)$, where the closure is taken in the Zariski topology. Of course, $C(X,Y)$ is an affine variety. We have the following.
\begin{lemma}\label{NormalTangent} For each $(x,x,w,v)\in C(X,Y)$, there are sequences $x^k\in X^0,\ y^k\in Y,\ \gamma^k\in\C$ and $\lambda^k\in \C^r$ such that 
\begin{itemize}
\item $x^k\to x$,
\item $y^k\to x$,
\item $w^k:=\gamma^k(x^k-y^k)\to w$,
\item $v^k:=\sum_{i=1}^r\lambda^k_id_{x^k} g_i\to v.$
\end{itemize}
\end{lemma}
\begin{proof} 
By construction, there are sequences $\bar x^k\in X,\ y^k\in Y,\ \bar\gamma^k\in \C$ and $\lambda^k\in \C^r$ such that $\bar x^k,y^k\to x,\ \bar w^k:=\bar\gamma^k(\bar x^k-y^k)\to w$ and $\sum_{i=1}^r\lambda^k_id_{\bar x^k}g_i\to v.$ It is clear that by taking subsequences if necessary, we may suppose that: 
\begin{itemize}
\item either $\bar x^k=y^k$ for every $k$ or $\bar x^k\not=y^k$ for every $k$,
\item for each $i$, either $\lambda^k_i\not=0$ for every $k$ or $\lambda^k_i=0$ for every $k$.
\end{itemize}
Set
$$\gamma^k=\begin{cases}0 \ \text{if }\ \bar x^k=y^k \text{ for every } k,\\
\bar\gamma^k \ \text{if }\ \bar x^k\not=y^k \text{ for every } k.
\end{cases}$$

Suppose that $\lambda^k_i\not=0$ for $i=1,\dots,r'\leqslant r,\ k\in \N$ and $\lambda^k_i=0$ for $i=r'+1,\dots,r,\ k\in \N$. Since $\bar x^k\in \overline X^0$, there exists a sequence $x^k\in X^0$ such that 
$$\|x^k-\bar x^k\|\leqslant\begin{cases}\frac{1}{k} \ \text{if }\ \bar x^k=y^k \text{ for every } k,\\
\frac{\|\bar x^k-y^k\|}{k} \ \text{if }\ \bar x^k\not=y^k \text{ for every } k,
\end{cases}$$ 
so $x^k\to x$. By continuity, we can also choose  $x^k$ so that 
$\|d_{x^k} g_i-d_{\bar x^k}g_i\|<\frac{1}{k\lambda^k_i}$ if $\lambda^k_i\not=0.$ Set $v^k:=\sum_{i=1}^r\lambda^k_id_{x^k}g_i$. Then
$$\begin{array}{llll}
\big\|v^k-\sum_{i=1}^r\lambda^k_id_{\bar x^k}g_i\big\|&=&\big\|\sum_{i=1}^{r'}\lambda^k_i\big(d_{x^k}g_i-d_{\bar x^k}g_i\big)\big\|\\
&\leqslant&\sum_{i=1}^{r'}|\lambda^k_i|\big\|d_{x^k}g_i-d_{\bar x^k}g_i\big\|<\frac{r'}{k}\to 0,\\
\end{array}$$
i.e., $v^k\to v.$ Set $w^k:=\gamma^k(x^k-y^k)$. Now if $\bar x^k=y^k$ for every $k$, then $\gamma^k=0$ and $w=\bar w^k=0$, so we have $w^k=0=w$. If $\bar x^k\not=y^k$ for every $k$, then
$$\|w^k-\bar w^k\|=|\gamma^k|\cdot\|(x^k-\bar x^k)\|\leqslant|\gamma^k|\cdot\frac{\|\bar x^k-y^k\|}{k}=\frac{\|\bar w^k\|}{k}\to 0.$$
Hence $w^k\to w.$ The lemma is proved.
\end{proof}

The following algebraic criterion permits us to check the Whitney regularity on $Y^0=Y\setminus V(p_{Y,W})$, where the notation $V(p_{Y,W})$ is from Proposition \ref{Affine}, and the affine set $W$ will be determined later.
\begin{lemma}\label{Criterionb1} Let $x\in Y^0.$ Then the pair $(X^0,Y^0)$ satisfies the Whitney condition (b) at $x$ if and only if 
for any $(x,x,w,v)\in C(X,Y)$, we have $ v\cdot w=0.$
\end{lemma}
\begin{proof} Suppose that $(X^0,Y^0)$ is Whitney regular at $x$ and assume for contradiction that there is $(x,x,w,v)\in C(X,Y)$ such that $v\cdot w\not=0.$ In view of Lemma \ref{NormalTangent}, there are  sequences $x^k\in X^0,\ y^k\in Y,\ \gamma^k\in\C$ and $\lambda^k\in \C^r$ such that 
\begin{itemize}
\item $x^k\to x$, $y^k\to x$,
\item $w^k:=\gamma^k(x^k-y^k)\to w$,
\item $v^k:=\sum_{i=1}^r\lambda^k_id_{x^k} g_i\to v.$
\end{itemize}
Note that $w\not=0$, so $w$ determines the limit of the sequence of secants $\overline{x^ky^k}$ and it follows that $x^k\not=y^k$ for $k$ large enough. By taking a subsequence if necessary, we may assume that $T_{x^k}X^0\to T.$ By assumption, $w\in T.$
For each $k$, let $\{b_1^k,\dots,b_r^k\}$ be an orthonormal basis of $N_{x^k}X^0$; recall that $N_{x^k}X^0:=\span\{d_{x^k}g_1,\dots,d_{x^k}g_r\}$ is the conormal space of $X^0$ at $x^k.$ Obviously $\langle\overline{b_1^k},\dots,\overline{b_r^k}\rangle^\perp={T_{x^k}X^0}$. By compactness, each sequence $b^k_i$ has an accumulation point $b_i$. Without loss of generality, suppose that $b_i^k\to b_i.$ It is clear that the system $\{b_1,\dots,b_r\}$ is also orthonormal and $\langle\overline b_1,\dots,\overline b_r\rangle^\perp=T.$ Let $\widetilde\lambda^k=(\widetilde\lambda^k_1,\dots,\widetilde\lambda^k_r)$ be such that $v^k:=\sum_{i=1}^r\widetilde\lambda^k_i b_i^k.$ Then $\widetilde\lambda^k$ is convergent to a limit $\widetilde\lambda$ and it is clear that $v=\sum_{i=1}^r\widetilde\lambda_i b_i.$ Finally, we have $w\in T=\langle\overline{b_1},\dots,\overline{b_r}\rangle^\perp\subset\langle\overline v\rangle^\perp,$ i.e., $v\cdot w=0$, which is a contradiction.

Now suppose that $v\cdot w=0$ for any $(x,x,w,v)\in C(X,Y)$ and assume,  that $(X^0,Y^0)$ is not Whitney regular at $x$. So there are  sequences $x^k\in X^0$ and $y^k\in Y^0$ with the following properties:
\begin{itemize}
\item $x^k\not=y^k,\ x^k\to x,\ y^k\to y$;
\item $T_{x^k}X^0\to T$;
\item the sequence of secants $\overline{x^ky^k}$ tends to a line $\ell\not\subset T$.
\end{itemize}
For each $k$, let $\{b_1^k,\dots,b_r^k\}$ be an orthonormal basis of $N_{x^k}X^0$ so $\langle\overline{b_1^k},\dots,\overline{b_r^k}\rangle^\perp={T_{x^k}X^0}$. As above, we may assume that $b_i^k\to b_i.$ Then the system $\{b_1,\dots,b_r\}$ is also orthonormal and $\langle\overline{b_1},\dots,
\overline{b_r}\rangle^\perp=T.$ Let $w^k:=\frac{x^k-y^k}{\|x^k-y^k\|}$; we can assume that the limit $w:=\lim w^k$ exists and clearly $w$ is a direction vector of $\ell$. By assumption, $w\not\in T=\langle\overline b_1,\dots,\overline b_r\rangle^\perp,$ i.e., there exists an index $j$ such that $b_j\cdot w\not=0.$ To get a contradiction, it is enough to show that there is a sequence $v^k:=\sum_{i=1}^r\lambda^k_id_{x^k} g_i$ such that $v^k\to b_j$, but this is clear since $b_j\in\span\{d_{x^k}g_1,\dots,d_{x^k}g_r\}$ so such a sequence always exists. The lemma is proved.
\end{proof}

Now according to \cite{Krick1991,Eisenbud1992,Greuel2008,radical}, it is possible to calculate a basis for the ideal $I(\Gamma_1)$ by calculating the radical  of the following ideal in $\C[x,y,w,v,\gamma,\lambda]$:
$$\left(\begin{array}{lcll}
 g_1(x)=\dots=g_r(x)=0\\
 g_1(y)=\dots=g_r(y)=\widetilde g_{r+1}(y)=\dots=\widetilde g_{p}(y)=0\\
 w=\gamma(x-y)\\
 v=\sum_{i=1}^r\lambda_id_x g_i\\
\end{array}\right).$$ 
Then by Buchberger's algorithm, we can calculate a Gr\"obner basis of $I(\Gamma_1)$. So in view of \cite[Theorem 5.1]{Jelonek2005}, \cite{Pauer1988}, we can compute a Gr\"obner basis of the ideal $I\big(C(X,Y)\big)$. Now we give another criterion for Whitney regularity.
\begin{lemma}\label{Criterionb2} Let $\{h_1(x,y,w,v),\dots,h_q(x,y,w,v)\}$ be a Gr\"obner basis of $I\big(C(X,Y)\big)$ and set
$$\Gamma_2:=\left\{\begin{matrix}
(x,x,w,v,\gamma,\lambda)\in\C^n\times\C^n\times\C^n\times\C^n\times\C\times\C:\\
h_1(x,x,w,v)=\dots=h_q(x,x,w,v)=0\\
\gamma\sum_{j=1}^nv_jw_j=1\\
\lambda p_{Y,\emptyset}(x)=1
\end{matrix}\right\},$$
where $p_{Y,\emptyset}(x)$ is the polynomial determined in Proposition \ref{Affine}.
Then the pair $(X^0,Y^0)$ is not Whitney regular at $x$ if and only if there exists $(w,v,\gamma,\lambda)\in\C^n\times\C^n\times\C\times\C$ such that $(x,x,w,v,\gamma,\lambda)\in\Gamma_2$.
\end{lemma}
\begin{proof} Note that $x\in Y^0$ if and only if $p_{Y,\emptyset}(x)\not=0$, i.e., there exists $\lambda\in\C$ such that $\lambda p_{Y,\emptyset}(x)=1$. In view of Lemma \ref{Criterionb1}, the pair $(X^0,Y^0)$ is not Whitney regular at $x$ if and only if there exist $w,v$ with $v\cdot w\not=0$ such that $(x,x,w,v)\in C(X,Y).$ The lemma follows easily.
\end{proof}

Now we determine an algebraic set $W=W(X,Y)$ in $Y$ with $\dim W<\dim Y$ and $V(p_{Y,\emptyset})\subset W$ such that the pair $(X^0,Y\setminus W)$ is  Whitney regular. Let $$\pi_2:\C^n\times\C^n\times\C^n\times\C^n\times\C\times\C\to \C^n$$ 
be the projection on the first $n$ coordinates. By Lemma \ref{Criterionb2}, $\pi_2(\Gamma_2)$ is the set of points where the Whitney condition (b) fails to be satisfied. By construction, $\overline{\pi_2(\Gamma_2)}^\Z$ is affine, where $\overline{\pi_2(\Gamma_2)}^\Z$ is the Zariski closure of $\pi_2(\Gamma_2)$. It follows from \cite{Whitney1965-1,Whitney1965-2} that $\dim \pi_2(\Gamma_2)<\dim Y$, so $\dim \overline{\pi_2(\Gamma_2)}^\Z<\dim Y$. Set 
$$W=W(X,Y):=\overline{\pi_2(\Gamma_2)}^\Z;$$ 
then obviously $\dim W<\dim Y$. Again, applying \cite{Krick1991,Eisenbud1992,Greuel2008}, \cite[Theorem 5.1]{Jelonek2005}, \cite{Pauer1988}, we can compute a Gr\"obner basis of the ideal $I(W).$ 

Finally, let 
\begin{itemize}
\item $X_0:=X,$ 
\item $X_1:=X_0\cap V(p_{X_0,\emptyset}),$
\item $X_2:=X_1\cap V(p_{X_1,W(X_0,X_1)}),$ 
\item $X_3:=X_2\cap V(p_{X_2,W(X_0,X_2)\cup W(X_1,X_2)}),$ $\dots,$
\item $X_i:=X_{i-1}\cap V(p_{X_{i-1},\bigcup_{j=0}^{i-2} W(X_j,X_{i-1})}),$ $\dots$
\end{itemize}
By induction, we can construct a finite filtration of algebraic sets $X=X_0\supset X_1\supset\dots\supset X_{n-r}\supset X_{n-r+1}=\emptyset$ with $\dim X_i>\dim X_{i+1}$. Let $B_i:=X_i\setminus X_{i+1}.$ Then $\S:=\{B_i\}_{i=1,\dots,q}$ is a Whitney stratification of $X$. Note that the degree of $X_i$ can be determined explicitly and depends only on $D.$


\section{Thom isotopy lemma for non-proper maps}\label{Isotopy}

We start this section with:

\begin{definition}
Let $f:X\to\C^m$ be a polynomial dominant map where $X$ is an algebraic set. Let  $\S=\{X_\alpha\}_{\alpha\in I}$ be  a stratification of $X.$ 
By $K_\infty(f|_{X_\alpha})$ we mean the set $\{ y\in\C^m: {\rm there \ is \ a \ sequence} \ x_n\to \infty; \ x_n\in X_\alpha: ||x_n||\nu(d_{x_n}(f|_{X_\alpha}))\to 0\ {\rm and} \ f(x_n)\to y\}$ (here $\nu$ denotes the Rabier function, for details see \cite{Jelonek2005}). 
Now let $C(f,X_\alpha)$ denote the set  of points where $f|_{X_\alpha}$ is not a submersion. By $\sing(f,\S)$ the set of stratified singular values of $f$, i.e.,
\begin{equation}\label{K0f}\sing(f,\S)=\bigcup_{\alpha\in I}K_0(f,X_\alpha), \end{equation}
where $K_0(f,X_\alpha)=\overline{f(C(f,X_\alpha))}.$ 
\end{definition}

By [Theorem 3.3, \cite{Jelonek2005}] we have that for every $\alpha$ the set  $K_\infty(f|_{X_\alpha})$ has measure $0$ in $\C^m.$
In particular the set $K(f)$ defined below has also measure $0.$

\begin{definition}
 Let $K(f)=K(f,\S)$ be the set of stratified generalized critical values of $f$ given by
\begin{equation}\label{Kf} K(f):=\bigcup_{\substack{\alpha\in I}} (K_0(f|_{X_\alpha})\cup K_\infty(f|_{X_\alpha})) \end{equation}
\end{definition}

Assume that $\S$ is an affine Whitney stratification of $X$, we prove that $K(f)$ contains the set of bifurcation values of $f$.
\begin{theorem}[First isotopy lemma for non-proper maps]\label{NonProperIsotopy} Let $X\subset\C^n$ be an affine variety with an affine Whitney stratification $\S$ and let $f:X\to\C^m$ be a polynomial dominant map. Let $K(f)$ be the set of stratified generalized critical values of $f$ given by  (\ref{Kf}). Then $f$ is locally trivial outside $K(f)$.
\end{theorem}
Before proving Theorem \ref{NonProperIsotopy}, recall that the Whitney condition (b) is equivalent to the condition (w) (see \cite[V.1.2]{Teissier1982}), so it is more convenient to use the condition (w) since we will need to construct rugose vector fields in the sense of \cite{Verdier1976}. In what follows, it is more convenient to work with the underlying real algebraic set of $X$ in $\R^{2n}$, denoted also by $X$; the affine Whitney stratification $\S$ of $X$ induces a semialgebraic Whitney stratification of the underlying set with the corresponding strata denoted by the same notations $X_\beta.$ We also identify the polynomial map $f$ with the real polynomial map $(\Re f_1,\dots,\Re f_m,\Im f_1,\dots,\Im f_m):X\to\R^{2m}.$ 

Let us recall the definitions pertaining to rugosity. Let $\varphi:X\to\R$ be a real function. We say that $\varphi$ is a \textbf{rugose function} if the following conditions are fulfilled:
\begin{itemize}
\item The restriction $\varphi|_{X_\beta}$ to any stratum $X_\beta$ is a smooth function.
\item For any stratum $X_\beta$ and for any $x\in X_\beta$, there exist a neighborhood $U$ of $x$ in $\C^{2n}$ and a constant $c>0$ such that for any $y\in X\cap U$ and $x'\in X_\beta\cap U$, we have $|\varphi(y)-\varphi(x')|\leqslant c\|y-x'\|$.
\end{itemize}
A \textbf{rugose map} is a map whose components are rugose functions. A vector field $v$ on $X$ is called a \textbf{rugose vector field} if $v$ is a rugose map and $v(x)$ is tangent to the stratum containing $x$ for any $x\in X$.

\begin{proof}[Proof of Theorem \ref{NonProperIsotopy}] 
Let $z\in\C^m\setminus K(f)$ where we identify $\C^m$ with $\R^{2m}$ and let $B$ be an open box centered at $z$ such that $\overline B\cap K(f)=\emptyset.$ To prove the theorem, it is enough to prove that $f$ is trivial on $B$. Without loss of generality, we may suppose that $z=0$ and $B=(-1,1)^{2m}$. Let $\partial_1,\dots,\partial_{2m}$ be the restrictions of the coordinate vector fields (on $\R^{2m}$) to $\overline B$. Set $U:=f^{-1}(\overline B),\ U_\beta=U\cap X_\beta$ and
$$I':=\{\beta\in I:\ X_\beta\cap U\not=\emptyset\}.$$ 
First of all, let us give a sufficient condition for trivializing a rugose vector field.
\begin{lemma}\label{Integrable} For $i=1,\dots,2m$, let $v_i$ be  vector fields on $X$ which are rugose in a neighborhood of $U$. Assume that $df(v_i)=\partial_i$ and there is a positive constant $c>0$ such that  $\|v_i(x)\|\leqslant\frac{\|x\|+1}{c}$ for any $x\in U$. Then $f$ is a topological trivial fibration over $\overline B.$
\end{lemma}
\begin{proof} It is enough to prove that there is a homeomorphism $\phi:f^{-1}(\overline B)\to f^{-1}(0)\times\overline B$ such that the following diagram commutes:
$$\begin{array}{rcl}
f^{-1}(\overline B)&\stackrel{\phi}{\longrightarrow}& f^{-1}(0)\times\overline B\\
f\searrow &&\swarrow \pi\\
&\overline B&
\end{array}$$
where $\pi$ denotes the projection on the second factor. We note the following facts:
\begin{itemize}
\item [(i)] The flow of $v_i$ preserves the stratification. 
This is a consequence of  rugosity.
\item [(ii)] For each $i$ and  any  $x\in U$, there is a unique integral curve of $v_i$ passing through $x$ (see \cite{Verdier1976}).
\end{itemize}
Set $Y^i_t:=(y_1,\dots,y_{i-1},t,y_{i+1},\dots,y_n)$ and $Y^i=\{Y^i_t:\ t\in[-1,1]\}.$
First of all, we will prove that the flow of $v_i$ induces a homeomorphism $\phi_i:f^{-1}(Y^i)\to f^{-1}(Y^i_0)\times[-1,1]$ such that the following diagram commutes:
$$\begin{array}{rcl}
f^{-1}(Y^i)&\stackrel{\phi_i}{\longrightarrow}& f^{-1}(Y^i_0)\times[-1,1]\\
p_i\circ f\searrow &&\swarrow \pi_i\\
&[-1,1]&
\end{array}$$
where $\pi$ denotes the projection on the second factor and $p_i$ denotes the projection on the $i^\text{th}$ coordinate. This follows from the following claim which states that there is no trajectory of $v_i$ going to infinity.
\begin{claim} For each $x\in f^{-1}(Y^i_0)$, let $\gamma$ be the integral curve of $v_i$ such that $\gamma(0)=x$. Then $\gamma$ reaches any level $f^{-1}(Y^i_t)$ at  time $t$ for $t\in[-1,1]$.
\end{claim}
\begin{proof} By assumption, $\|\dot{\gamma}(t)\|\leqslant\frac{\|\gamma(t)\|+1}{c}$. Without loss of generality, suppose that $t>0$. In light of the Gronwall Lemma, by repating the calculation of \cite[Theorem 3.5]{DAcunto2005}, we get
\begin{eqnarray*}
\|\gamma(t)\| & \leqslant & \|\gamma(0)\|+\int_0^t\frac{\|\gamma(s)\|+1}{c}ds \\
& = & \|x\|+\frac{t}{c}+\int_0^t\frac{\|\gamma(s)\|}{c}ds \\
& \leqslant & \Big(\|x\|+\frac{t}{c}\Big)\exp\int_0^t\frac{ds}{c}=\Big(\|x\|+\frac{t}{c}\Big)e^\frac{t}{c}<+\infty,
\end{eqnarray*}
which implies that the trajectory $\gamma$ does not go to infinity at  time $t$. The claim follows.
\end{proof}
For any $x\in f^{-1}(Y^i_{0})$, let $h_i(x,t)=x+\int_0^t\stackrel{.}{\gamma}(s)ds$. Then $h_i$ defines a homeomorphism $f^{-1}(Y^i_0)\times[-1,1]\to f^{-1}(Y^i)$. Then $\phi_i=h_i^{-1}$ is the required homeomorphism. 

Now for $x\in f^{-1}(0)$, let
$h:f^{-1}(0)\times\overline B\to f^{-1}(\overline B)$ be defined by
$$h(x,t_1,\dots,t_{2m})=h_{2m}(\dots(h_2(h_1(x,t_1),t_2),\dots,t_{2m}).$$
Then $\phi:=h^{-1}$ is a homeomorphism, as required. The lemma is proved.
\end{proof}

For each $\beta\in I'$, it is clear that $f|_{X_\beta}$ is a submersion on $(f|_{X_\beta})^{-1}(\overline B)$, so for $x\in U_\beta,$ the differential $d_x(f|_{X_\beta}):T_x X_\beta\to\R^{2m}$ is surjective, which induces an isomorphism of vector spaces 
$$\widetilde{d_x}(f|_{X_\beta}):T_x X_\beta/\ker d_x(f|_{X_\beta})\cong\R^{2m}.$$
Thus for each $i=1,\dots,2m$, the vector field $\partial_i$ can be lifted uniquely and smoothly on each stratum $X_\beta$ with $\beta\in I'$, to the vector field called the horizontal lift of $\partial_i$ and denoted  by $v_i^\beta.$  Clearly, $v_i^\beta(x)$ is the unique vector in $T_x X_\beta$ which  lifts  $\partial_i$ and is orthogonal to $\ker d_x(f|_{X_\beta})$. Each $v_i^\beta$ has the following important properties.
\begin{lemma}\label{Bounded} Let $c>0$ be such that $(\|x\|+1)\nu\big(d_x(f|_{X_\beta})\big)\geqslant c$ for any $\beta\in I'$ and any $x\in U_\beta$; recall that $\nu$ is the Rabier function \cite{Rabier1997}. For each $x\in X_\beta$ with $\beta\in I'$, we have 
$$\|v_i^\beta(x)\|\leqslant\frac{\|x\|+1}{c}.$$
\end{lemma}
\begin{proof} Let $\B_\beta$ be the closed unit ball centered at the origin in $T_x X_\beta$. Then $d_x(f|_{X_\beta})(\B_\beta)$ is an ellipsoid in $\R^{2m}$ with $\nu\big(d_x(f|_{X_\beta})\big)$ as the length of shortest semiaxis. Let $\B$ be the closed unit ball centered at the origin in $\R^{2m}$. Then $\big(\widetilde d_x(f|_{X_\beta})\big)^{-1}\Big({\nu\big(d_x(f|_{X_\beta})\big)}\B\Big)$ is an ellipsoid in $T_x X_\beta$ with $1$ as the lenght of longest semiaxis. Therefore the longest semiaxis of the ellipsoid $\big(\widetilde d_x(f|_{X_\beta})\big)^{-1}(\B)$ is $1/{\nu\big(d_x(f|_{X_\beta})\big)}$. Consequently,
$$\|v_i^\beta(x)\|\leqslant\frac{1}{\nu\big(d_x(f|_{X_\beta})\big)}\leqslant\frac{\|x\|+1}{c},$$
which yields the lemma.
\end{proof}

\begin{lemma}\label{Compare} Let $x\in U$, let $X_\beta$ be the stratum containing $x$ and let $X_\alpha$ be a stratum such that $X_\beta\subset\overline X_\alpha$. Then there exists an open neighborhood $W$ of radius not greater than $1$ of $x$ such that for all $y\in W\cap X_\alpha$, we have 
$$\|v_i^\alpha(y)\|< 2\|v_i^\beta(x)\|,$$
for $i=1,\dots,2m.$
\end{lemma}
\begin{proof} 
Assume for contradiction that there are an index $i_0$, a stratum $X_{\alpha_0}$ and a sequence $x^k\in X_{\alpha_0}$ such that $x^k\to 0$ and 
 $\|v_i^{\alpha_0}(x^k)\|\geqslant 2\|v_i^\beta(x)\|$. Taking a subsequence if necessary, we may assume that $T_{x^k}X_{\alpha_0}\to T$. Since the stratification is Whitney regular, it is (a) Whitney regular, which yields $T\supset T_xX_\beta$. The following claims are straightforward.
\begin{claim}\label{Claim1} Let $L:H\to\R^m$ be a surjective linear map and let $\widetilde L:H/\ker L\cong \R^m$ be the induced linear isomorphism. Let $H'\subset H$ be a linear subspace and assume that $L|_{H'}$ is also surjective. Then for any $w\in\R^m$, we have 
$$\|(\widetilde {L})^{-1}(v)\|\leqslant\|(\widetilde {L|_{H'}})^{-1}(v)\|.$$
\end{claim}
\begin{claim}\label{Claim2} The sequence $v_i^{\alpha_0}(x^k)$ is convergent.
\end{claim}
By Claim \ref{Claim2}, let $w_i:=\lim_{k\to\infty}v_i^{\alpha_0}(x^k)$. Then it is clear that $\|w_i\|\geqslant 2\|v_i^\beta(x)\|$ and $w_i=(\widetilde{d_x}f|_T)^{-1}(\partial_i)$ where $\widetilde{d_x}f|_T$ is given by the linear isomorphism $T/\ker (d_xf|_T)\cong\R^{2m}$. Since $T\supset T_xX_\beta$, it follows from Claim \ref{Claim1} that 
$$\|w_i\|\leqslant\|(\widetilde{d_x}f|_{T_xX_\beta})^{-1}(\partial_i)\|=\|v_i^\beta(x)\|.$$
This is a contradiction, which ends the proof of the lemma.
\end{proof}

Note that, for fixed $i$, the vector field on $U$ which coincides with $v_i^\beta$ on each $U_\beta$ is not necessarily a smooth vector field. In what follows, we will try to deform these fields to produce a rugose vector field in the sense of \cite{Verdier1976}, which satisfies the assumption of Lemma \ref{Integrable}. The process is carried out by induction on dimension.

For $2m\leqslant d\leqslant 2\dim_\C X$, let $I'_d:=\{\beta\in I':2m\leqslant\dim X_\beta\leqslant d\},\ B_d:=\bigcup_{i\in I'_{d}}X_\beta$ and $U_d=B_d\cap U$. By induction on $d$, we construct a rugose vector field on a neighborhood of $U_{2\dim_\C X}$ in $X$ with the  property of Lemma \ref{Integrable}. For $d=2m$, let $v_i^{2m}$ be the restriction to an open neighborhood of $U_{2m}$ in $X$ of the smooth vector field on $B_{2m}$ which coincides with each $v_i^\beta$ on $X_\beta$ for $\beta\in I'_{2m}$. Then $v_i^{2m}$ is clearly rugose, $df(v_i^{2m})=\partial_i$ and by Lemma \ref{Bounded}, $\|v_i^{2m}(x)\|\leqslant\frac{\|x\|+1}{c}$ for any $x\in U_{2m}$.

For each $i$, assume that we have constructed a rugose vector field, denoted by $v_i^d$, on a neighborhood $\widetilde U_d$ of $U_d$ in $B_d$ such that $d_xf\big(v_i^{d}(x)\big)=\partial_i$ and  $\|v_i^d(x)\|\leqslant\frac{\|x\|+1}{c_d}$ for every $x\in \widetilde U_d$, where $c_d$ is a positive constant. We need to extend each $v_i^d$ to a rugose vector field $v_i^{d+2}$ on a neighborhood of $U_{d+2}$ in $B_{d+2}$ such that $\|v_i^{d+2}(x)\|\leqslant\frac{\|x\|+1}{c_{d+2}}$ for every $x\in U_{d+2}$, where $c_{d+2}$ is also a positive constant (recall that the strata of $\S$ have even dimension). Note that to construct $v_i^{d+2}$, it is enough to construct $v_i^{d+2}$ separately on each stratum $X_\alpha$ with $\alpha\in I'_{d+2}\setminus I'_d.$ Without loss of generality, suppose that $I'_{d+2}\setminus I'_d=\{\alpha\}.$ 

By similar arguments as in \cite{Verdier1976}, for each $i=1,\dots,2m$, there is a rugose vector field on a neighborhood $\widetilde U_{d+2}$ of $U_{d+2}$ in $B_{d+2}=B_d\cup X_\alpha$, denoted by $w_i^{d+2}$, which extends $v_i^d$, such that:
\begin{itemize}
\item [(i)] The restriction $w_i^{d+2}|_{U_\alpha}$ is a smooth vector field.
\item [(ii)] For $x\in U_\alpha$, we have $d_xf\big(w_i^{d+2}(x)\big)=\partial_i.$
\end{itemize}
For each $x\in U_d$, let $X_\beta$ be the stratum containing $x$, and let $W_x$ be a neighborhood of $x$ given by Lemma \ref{Compare}. Since $w_i^{d+2}$ is continuous, by shrinking $W_x$ if necessary, we may assume that 
\begin{equation}\label{F1} \|w_i^{d+2}(y)\|<2\|v_i^{d}(x)\|,
\end{equation}
for any $y\in W_x$. Let $V_d:=\bigcup_{x\in U_d}W_x$, then $V_d$ is an open neighborhood of radius not bigger than $1$ of $U_d$. By a smooth version of Urysohn's lemma, there is a smooth function $\varphi:\R^{2n}\to[0,1]$ such that $\varphi^{-1}(0)=\R^{2n}\setminus V_d$ and $\varphi^{-1}(1)=U_d$. For $x\in \widetilde U_{d+2}$, set
$$v_i^{d+2}(x):=\left\{
\begin{array}{llll}
w_i^{d+2}(x)=v_i^d(x) & \text{ if } x\in \widetilde U_{d+2}\cap\widetilde U_d \\
\big(1-\varphi(x)\big)v_i^\alpha(x)+\varphi(x)w_i^{d+2}(x) & \text{ if } x\in \widetilde U_{d+2}\setminus\widetilde U_d.\\
\end{array}\right.$$
Clearly, the restriction of $v_i^{d+2}$ on each stratum is a smooth vector field. Moreover for $x\in \widetilde U_{d+2}\setminus\widetilde U_d$, we have
$$\begin{array}{lllll}
d_xf\big(v_i^{d+2}(x)\big)&=d_xf\Big(\big(1-\varphi(x)\big)v_i^\alpha(x)+\varphi(x)w_i^{d+2}(x)\Big)\\
&=\big(1-\varphi(x)\big)d_xf\big(v_i^\alpha(x)\big)+\varphi(x)d_xf\big(w_i^{d+2}(x)\big)\\
&=\big(1-\varphi(x)\big)\partial_i+\varphi(x)\partial_i=\partial_i.
\end{array}$$
Let us prove that $v_i^{d+2}$ is a rugose vector field. For any $x\in \widetilde U_{d+2}\cap\widetilde U_{d}$, let $X_\beta$ be the stratum containing $x$.   
For $x'\in W_x\cap X_\beta$ and $y\in W_x\cap X_\alpha$ with $\beta\in I'_d$, we have
$$\begin{array}{lllll}
\|v_i^{d+2}(y)-v_i^{d+2}(x')\|&=\big\|\big(1-\varphi(y)\big)v_i^\alpha(y)+\varphi(y)w_i^{d+2}(y)-v_i^d(x')\big\|\\
&=\big\|\big(1-\varphi(y)\big)v_i^\alpha(y)-\big(1-\varphi(y)\big)w_i^{d+2}(y)+w_i^{d+2}(y)-v_i^d(x')\big\|\\
&\leqslant\big(1-\varphi(y)\big)\|v_i^\alpha(y)-w_i^{d+2}(y)\|+\|w_i^{d+2}(y)-v_i^d(x')\|\\
&\leqslant\big(1-\varphi(y)\big)(\|v_i^\alpha(y)\|+\|w_i^{d+2}(y)\|)+\|w_i^{d+2}(y)-v_i^d(x')\|.
\end{array}$$
We note the following facts:
\begin{itemize}
\item Since $1-\varphi(y)$ is a smooth function, it is locally Lipschitz; with no loss of generality, assume that $1-\varphi(y)$ is Lipschitz 
on $W_x$ with constant $c_1.$ Then $$1-\varphi(y)=\big(1-\varphi(y)\big)-\big(1-\varphi(x')\big)\leqslant c_1\|y-x'\|.$$
\item By Lemma \ref{Compare} and by the continuity of $w_i^{d+2}$, there is a positive constant $c_2$ depending only on $x$ such that $\|v_i^\alpha(y)\|+\|w_i^{d+2}(y)\|\leqslant c_2$\\ (we can take $c_2:=\max\big\{2\|v_i^\beta(x)\|,\sup_{y\in \overline W_x\cap X_\alpha}\|w_i^{d+2}(y)\|\big\}$).
\item Since $w_i^{d+2}$ is rugose, it follows that there is a positive constant $c_3$ depending only on $x$ such that $\|w_i^{d+2}(y)-v_i^d(x')\|\leqslant c_3\|y-x'\|.$
\end{itemize}
Hence 
$$\|v_i^{d+2}(y)-v_i^{d+2}(x')\|\leqslant (c_1c_2+c_3)\|y-x'\|,$$
i.e., $v_i^{d+2}$ is rugose. Now it remains to show that there is a positive constant $c_{d+2}$ such that $\|v_i^{d+2}(y)\|\leqslant\frac{\|y\|+1}{c_{d+2}}$ for every $y\in \widetilde U_{d+2}$. Obviously, the statement holds for $y\in(\widetilde U_{d+2}\cap\widetilde U_d)$ by the induction assumption and for $y\in(\widetilde U_{d+2}\setminus V)$ by Lemma \ref{Bounded}, so we can suppose that $y\in (V\cap \widetilde U_{d+2})\setminus \widetilde U_d$, which clearly implies that $y\in X_\alpha$. By construction and by Lemma \ref{Compare}, there are a point $x\in U_d$ and an open neighborhood $W_x$ of radius not greater than $1$ of $x$ containing $y$ such that $\|v_i^\alpha(y)\|<2\|v_i^\beta(x)\|$, where $\beta$ is the index of the stratum $X_\beta$ containing $x$. By Lemma \ref{Bounded}, it follows that
\begin{equation}\label{F2}\|v_i^\alpha(y)\|<2\frac{\|x\|+1}{c}\leqslant 2\frac{\|y\|+\|x-y\|+1}{c}\leqslant 2\frac{\|y\|+2}{c}\leqslant 4\frac{\|y\|+1}{c},\end{equation}
where $c$ is the constant in the same lemma. Similarly, in view of (\ref{F1}) and the induction assumption, we have 
\begin{equation}\label{F3}\|w_i^{d+2}(y)\|<2\|v_i^{d}(x)\|\leqslant 2\frac{\|x\|+1}{c_d}\leqslant 2\frac{\|y\|+\|x-y\|+1}{c_d}\leqslant 2\frac{\|y\|+2}{c_d}\leqslant 4\frac{\|y\|+1}{c_d}.\end{equation}
Thus (\ref{F2}) and (\ref{F3}) yield
$$\begin{array}{llll}
\|v_i^{d+2}(y)\| &=& \big|\big(1-\varphi(y)\big)v_i^\alpha(y)+\varphi(y)w_i^{d+2}(y)\big\|\\
&\leqslant& \big(1-\varphi(y)\big)\|v_i^\alpha(y)\|+\varphi(y)\|w_i^{d+2}(y)\|\\
&\leqslant& \big(1-\varphi(y)\big)4\frac{\|y\|+1}{c}+\varphi(y)4\frac{\|y\|+1}{c_d}\\
&<&\Big(\frac{4}{c}+\frac{4}{c_d}\Big)(\|y\|+1).
\end{array}$$
Set $c_{d+2}=\min\Big\{\frac{1}{\frac{4}{c}+\frac{4}{c_d}},c,c_d\Big\}$, then $\|v_i^{d+2}(y)\|\leqslant\frac{\|y\|+1}{c_{d+2}}$ for every $y\in \widetilde U_{d+2}$. By induction, there exists a rugose vector field on a neighborhood of $U_{2\dim_\C X}$ in $X$ with the  property of Lemma \ref{Integrable}. Then the theorem follows by applying Lemma \ref{Integrable}.
\end{proof}

The following corollary follows immediately from Theorem \ref{NonProperIsotopy}.
\begin{corollary}\label{NonProperIsotopySubmersion} Let $X\subset\C^n$ be an affine variety with an affine Whitney stratification $\S$ and let $f:X\to\C^m$ be a polynomial dominant map. Assume that for any stratum $X_\beta\in\S$, the restriction $f|_{X_\beta}$ is a submersion and 
 $K_\infty(f|_{X_\beta})=\emptyset$. Then $f$ is a locally trivial fibration.
\end{corollary}


\section {Computation of the sets of stratified generalized critical values}\label{StratifiedKf}

In this section, we will compute the set $K(f)$ of stratified generalized critical values  of $f$, for which we need to construct an affine Whitney stratification of $X$ and then apply \cite{Jelonek2005} for each stratum of this stratification. The process is a bit different from the construction in Subsection \ref{SectionWhitneyb} since we only need to construct  such an affine Whitney stratification ``partially", by remarking the following facts:
\begin{itemize}
\item As the construction of Whitney stratifications is by  induction on dimension, we only need to proceed until the dimension shrinks below $m$ since the restriction of $f$ to any stratum of dimension $<m$ is always singular. 
\item For any algebraic set $Z\subseteq X$, let 
$$\displaystyle r_Z:=\max_{x\in Z\setminus V(p_{Z,\emptyset})}\rank Jac_x(f|_Z)\ \text{ and }\ H(Z):=\overline{\{x\in Z\setminus V(p_{Z,\emptyset}):\rank Jac_x(f|_Z)<r_Z\}}^\Z.$$ 
Then at any step of the induction process, the construction in Subsection \ref{SectionWhitneyb} can be omitted if $r_Y<m.$
\end{itemize}
Let us now construct such a stratification. 
With the same notations as in Lemma \ref{Criterionb2}, let 
$$\Gamma_3:=\bigcup_{k=1}^t\left\{\begin{matrix}
(x,x,w,v,\gamma,\lambda,\mu)\in\C^n\times\C^n\times\C^n\times\C^n\times\C\times\C\times\C^{t}:\\
h_1(x,x,w,v)=\dots=h_q(x,x,w,v)=0\\
\gamma\sum_{j=1}^nv_jw_j=1\\
\lambda p_{Y,\emptyset}(x)=1\\
\mu_k M^{(m,p)}_k(x)=1
\end{matrix}\right\},$$
where each $M^{(m,p)}_k(x)$ is a minor of the  matrix
$$A(x):=\begin{bmatrix} 
d_x f_1\\
\vdots\\
d_x f_m\\
d_x g_1\\
\vdots\\
d_x g_r\\
d_x\widetilde g_{r+1}\\
\vdots\\
d_x\widetilde g_p
\end{bmatrix},$$
obtained by deleting $n-m-p$ columns. So $\Gamma_3$ differs from $\Gamma_2$ in the last $t$ equations since we are only interested in finding the points where the Whitney condition (b) is not satisfied, outside $P(Y,\emptyset)$. Let $$\pi_3:\C^n\times\C^n\times\C^n\times\C^n\times\C\times\C\times\C^{t}\to \C^n$$ 
be the projection on the first $n$ coordinates. By Lemma \ref{Criterionb2}, $\pi_3(\Gamma_3)$ is the set of points where the Whitney condition (b) fails. Obviously $\pi_3(\Gamma_3)\subset\reg(f|_{Y\setminus P(Y)})$ and $\dim \pi_3(\Gamma_3)<\dim Y.$ Set $\widetilde W:=\overline{\pi_3(\Gamma_3)}^\Z$. Then obviously $\dim \widetilde W<\dim Y$. Again, we can compute a Gr\"obner basis of the ideal $I(\widetilde W).$

Finally, set 
\begin{itemize}
\item $X_0:=X,$ 
\item $X_1:=X_0\cap V(p_{X_0,\emptyset}),\ S_1=K_0(f|_{X_0\setminus X_1}), \dots,$
\item $X_i:=X_{i-1}\cap V(p_{X_{i-1},\bigcup_{j=0}^{i-2} \widetilde W(X_j,X_{i-1})}),\ S_i=K_0(f|_{X_{i-1}\setminus X_i}), \dots$ 

\end{itemize}
By induction, we can construct a finite filtration of algebraic sets $X=X_0\supset X_1\supset\dots\supset X_q\supset X_{q+1}\supseteq\emptyset$ with $\dim X_i>\dim X_{i+1}$ and $r_{X_{q+1}}<m$. It is clear that this filtration does not induce an affine Whitney stratification of $X$. However, it shows that there is an affine Whitney stratification $\S$ such that
$$\sing(f,\S)=\bigcup_{i=1}^{q} S_i\cup \overline{f(X_{q+1})}.$$
Let
$B_i:=X_i\setminus X_{i+1}.$ 
Then $\{B_i\}_{i=0,\dots,q}$ is an affine Whitney stratification of $X\setminus X_{q+1}.$ Every variety $B_i$ can be realized as a closed affine variety in $\C^{n+1}$, by the embedding $B_i\ni x\mapsto \big(x, 1/P_{X_i,\bigcup_{j=0}^{i-1}\widetilde W(X_j,X_{i})}(x)\big)\in \C^{n+1}$ for $i>0$ or the embedding $B_0\ni x\mapsto \big(x, 1/P_{X_0,\emptyset}(x)\big)\in \C^{n+1}$. Let $K_\infty(f|_{B_i})$ be the set of asymptotic critical values of $f|_{B_i}$, which now can be computed analogously as in \cite{Jelonek2005} - this will be done in the next section. Then from the construction, it is clear that the set of stratified generalized critical values of $f$ is given by
$$K(f):=\bigcup_{i=1}^{q} K_\infty(f|_{B_i})\cup \sing(f,\S),$$
and $K(f)$ can be computed effectively.
Note that Remark \ref{uwaga} and  elementary properties of Gr\"obner basis implies:

\begin{corollary}
Let $X\subset \C^n$ be an affine variety of pure dimension and let $f=(f_1,\dots, f_m) :X\to \C^m$ be a polynomial mapping. Let $\Bbb F\subset \C$ be a subfield generated by coefficients of generators of $I(X)$ and all coefficients of polynomials $f_i, \ i=1,\dots,m.$ Then there is a nowhere dense affine variety $K(f)\subset \C^m$, which is described by polynomials from $\Bbb F[x_1,\dots,x_m]$, such that all bifurcation values $B(f)$ of $f$ are contained in $K(f).$ In particular, for $m=1$, if $X$ and $f$ are described by polynomials from $\Bbb Q[x_1,\dots,x_n]$, then all bifurcation values of $f$ are algebraic numbers. 
\end{corollary}

\section{Computation of $K_0(f|_{B_i})\cup K_\infty(f|_{B_i})$}\label{B_i}

Let $k=\R$ or $k=\C.$
Let $X\cong k^n$, $Y\cong k^m$
be finite dimensional vector spaces (over $k$). We consider those
space equipped with
the canonical scalar (hermitian) products.
Let us denote by
${\mathcal L}(X, Y)$ the set of linear mappings from $X$ to $Y$ and by
$\Sigma =\Sigma(X,Y)\subset {\mathcal L}(X, Y)$ the set of non-surjective
mappings. In
this section we give several different expressions for a distance of an
$A\in {\mathcal L}(X, Y)$ to the space $\Sigma $   of singular operators.
Let us recall
the  first following (\cite{Rabier1997}):

\begin{definition}\label{odwrotne}
Let $A\in {\mathcal L}(X, Y)$. Set
$$\nu (A) = {\rm inf}_{ ||\phi||=1} ||A^* (\phi)||,$$
where $A^*: {\mathcal L}(Y^*, X^*)$ is adjoint operator and $\phi\in Y^*.$
\end{definition}

Let $\alpha, \beta :{\mathcal L}(X, Y) \to \R_+ $ be two  non-negative
functions. We shall
say that
$\alpha$ and $\beta$ are {\it equivalent} (we write $\alpha \sim
\beta$) if there are
constants
$c,d >0$ such that
$$
c\alpha(A) \le \beta (A) \le d \alpha (A)$$
for any $A\in {\mathcal L}(X, Y)$.
We shall give below several functions  equivalent  to $\nu$.
Let $A= (A_1,\dots, A_m)\in {\mathcal L}(X, Y)$ and let $\overline{A_i}= \grad \ A_i.$
Denote by $<(\overline{A_j})_{j\not= i}>$ the linear space generated by vectors
$(\overline{A_j}), \, j\not= i$.
Let
$$\kappa (A) = {\rm min}_{1\le i\le m} \dist(\overline{A_i},
      <(\overline{A_j})_{j\not= i}>),$$
be the {\it Kuo number of} $A$.

\begin{proposition}[\cite{kos}]\label{kuo} The Kuo function $\kappa$ is
equivalent to $\nu$
of Rabier. More precisely
$$\nu (A) \le \kappa (A)\le \sqrt{m} \nu (A).$$
\end{proposition}


\begin{definition}
Let $A\in {\mathcal L}(X, Y)$ and let $H\subset X$ be a linear subspace.
We set
$$\nu (A, H) = \nu(A|_H),\ \kappa(A, H) =\kappa(A|_H),$$
where $A|_H$ denotes the restriction of $A$ to $H$.
\end{definition}

From Proposition \ref{kuo} we get immediately the following corollary.

\begin{corollary}
We have $\nu (A, H) \sim \kappa(A, H).$
\end{corollary}

In fact we have also the following explicit expression for $\kappa(A, H)$ (see \cite{Jelonek2005}.

\begin{proposition}
Let $A= (A_1,\dots, A_m)\in {\mathcal L}(X, Y)$ and let $H\subset X$ be a linear
subspace. Assume that $H$ is given by a system of linear equations
$B_j=0, j=1,\dots,r$.
Then
$$\kappa (A, H) = {\rm min}_{1\le i\le m} \dist(\overline{A_i},
<(\overline{A_j})_{j\not= i};
(\overline{B_j})_{j=1,\dots,r}>),$$ where $\overline{A_i}=\grad \ A_i$
and $\overline{B_j}=\grad \
B_j.$
\end{proposition}

Finally we introduce we function $g'$  which will be useful in the explicit description of the set of generalized critical values.

\begin{definition}\label{gieprime} Let $A\in {\mathcal L}(k^n, k^m)$,
where $n\ge m+r$, and let $H\subset k^n$ be a linear
subspace given by a system of independent linear equation
$B_l=\sum b_{lk} x_{k}, \, l=1, \dots, r.$
By abuse of notation we denote by $A$   the matrix (in the canonical
bases in $k^n$ and
$k^m$) of the mapping $A$.
Let $ C=$ be a $(m+r)\times n $ matrix given by
rows $A_1,\dots,A_m; B_1,\dots,B_r$ (we identify $A_i=\sum a_{ik} x_k$
with the vector
$(a_{j1},\dots, a_{jn})$, similarly for $B_l$). Let $M_{I},$
where $I= (i_1,\dots,i_{m+r})$,
denote a $((m+r)\times (m+r))$ minor
of $C$ given by columns indexed by $I$.
Let $M_J(j)$ denote a $(m+r-1)\times (m+r-1)$ minor
given by columns indexed by $J$
and  by deleting the
$j^\text{th}$ row , where $1\le j\le m.$ Note that we delete only $A_j$ rows! We set
$$g'(A, H)=\max_I \Big\{\min_{\{ J\subset I, \ 1\le j\le m\}}
\frac{|M_I|}{|M_J(j)|}\Big\},$$
(where we consider only numbers with $M_J(j)\not=0$, if all numbers
$M_J(j)$ are zero, we put $g'(A)=0$).
\end{definition}

In particular we have the following (see \cite{Jelonek2005}).

\begin{proposition}\label{nuprime}
We have  $g'(A, H)\sim \nu (A, H).$
\end{proposition}

Now we can prove the following theorem.

\begin{theorem}\label{glowne}
Let $B_i$ be a strata of $X$ as in Section \ref{StratifiedKf}. Then the set $K(f|_{B_i})=K_0(f|_{B_i})\cup K_\infty(f|_{B_i})$ is a nowhere dense  algebraic subset 
of $\C^m.$ 
\end{theorem}

\begin{proof}
It is standard fact that $K_0(f|_{B_i})$ is algebraic and nowhere dense (for details see the end of  subsection \ref{subsection}).
Hence it is enough to focus on $K_\infty(f|_{B_i}).$

By  construction the set $X:=B_i\subset \C^{n+1}$ 
is a complete intersection. 
Let us recall notation of Definition \ref{gieprime}. For $x \in \C$
let $A=d_x f$, and
$B_l = d_x b_l $, $l=1,\dots , r$.
Let $A\in {\mathcal L}(k^n, k^m)$, where $n\ge m+r$,
and let $T_xX =H\subset k^n$ be a linear subspace given by a system
of independent linear equation
$B_l=\sum b_{lk} x_{k}, \, l=1, \dots, r.$
By abuse of notation we denote by $A$   the matrix (in the canonical
bases in $k^n$ and
$k^m$) of the mapping $A$.
Let $ C=$ be a $(m+r)\times n $ matrix given by
rows $A_1,\dots,A_m; B_1,\dots,B_r$ (we identify $A_i=\sum a_{ik} x_k$
with the vector
$(a_{j1},\dots, a_{jn})$, similarly for $B_l$).

For an index $I=(i_1,\dots,i_{m+r})\subset \{1,\dots, n\}$
      let $M_{I}(x)$
denote the $((m+r)\times (m+r))$ minor
of $C$ given by columns indexed by $I$.
      For integers
$j\in I, 1\le k\le m$ we denote by
      $M_{I(k,j)} (x)$ the
      $(m+r-1)\times (m+r-1)$ minor obtained by deleting $j^\text{th}$ column and
$k^\text{th}$ row.
Note that we delete only
$A_k$, $1\le k\le m$ rows!

Hence $M_{I}$ and $M_{I(k,j)}$ are regular (restriction of
polynomials) functions on $X$.
We define now a family of  rational functions on $X$:
$$
W_{I(k,j)}(x)=M_I(x)/M_{I(k,j)}(x)$$
where   for
$M_{I(k,j)}\equiv 0$,  we put $W_{I(k,j)}\equiv0$.
We write  $b =(b_1,\dots,b_r)$ and $(f,b):\C^{n+1} \to \C^m\times \C^r$,
here we consider
$f_1,\dots,f_m,$ and $b_1,\dots ,b_r$ as polynomials on $\C^{n+1}$ (note that these polynomials does not depend on variable $x_{n+1}$).

Let $s=  \binom n  {m+r}$
and let $M_{I_1},\dots,M_{I_s}$ be all possible main minors of a matrix of
$d_x ( f,b).$ For every index $I_l$ take a pair $(k_l,j_l)$ which determine a
$(m+r-1)\times(m+r-1)$ minor of $M_{I_l}$ (we consider here only mniors which are not identically zero). We denote a  sequence
$(k_1,j_1),\dots,
(k_s,j_s)$ by $(k,j) \in \N^s \times \N^s$ and
we consider a rational function:
$$\Phi_{(k,j)}= \Phi((k_1,j_1),\dots,(k_s,j_s))
      : X
\to \C^m\times\C^N
$$
where the first component of $\Phi_{(k,j)}$ is $f$ and next components are
$W_{I_i(k_i,j_i)}$, $i=1,\dots, s$ and all products
$x_lW_{I_i(k_i,j_i)}$, $i=1,\dots, s$; $l=1,\dots, n$.

We can assume that for some choice of $l$ we have
$W_{I_l(k_l,j_l)}\not\equiv 0,$
and consequently dim $cl(\Phi_{(k,j)}(X))={\rm dim }\ X =\ n-r.$ Here
$cl(Y)$ stands for
the closure
of $Y$ in the strong (or which is the same, in the Zariski topology). Let
$\Gamma(k,j)= cl( \Phi_{(k,j)}(X))$ (by  $\Phi_{(k,j)}(X)$ we mean the set $\Phi_{(k,j)}(X\setminus P)$, where $P$ is a set of poles of $\Phi_{(k,j)}$).

Now for a given $q\in \{1,\dots,n\}$, consider the set $B_{i,q}:=B_i\setminus \{ x_q=0\}$ and the embedding $\iota: B_{i,q}\ni x\mapsto (x, 1/x_{q})\in \C^{n+2}.$ Finally let $\Phi_{(k,j),q}(x, x_{n+2}):=(\Phi_{(k,j)}(x),x_{n+2})$ and $\Gamma((k,j),q):=cl(\Phi_{(k,j),q}(X)).$

Let us recall that $y \in K_\infty (f)$ if  there exists
      a  sequence $ x\rightarrow\infty$  such that
$$ f(x)\rightarrow y \, \,  \text{and} \ \| x\| g'(x) \rightarrow 0,$$
were $g'(x) =g' (d_x f,T_x X)$
We have
\begin{lemma}\label{FormulaKf}
$$ K_\infty (f)
= \C^m \cap \bigcup_{(k,j),q}
\Gamma((k,j),q),$$ where we identify
$\C^m$ with $\C^m\times (0,\dots,0).$
\end{lemma}
\begin{proof}

Let
$y\in K_\infty(f).$ Hence  there is a
sequence $x^l\to \infty$, such that $f(x^l)\to y$ and $\| x^l\| g'(x_l) \rightarrow 0$. 
Moreover, if $x=(x_1,\dots,x_n)$, then there is at least one $q; 1\le q\le n$ such that $x^l_q\to \infty.$
If $\{x_l,\ l=1,2,\dots\} \subset C(f)$ ($C(f)$ denotes the set of critical points of $f$), then it is easy to see that $y\in \C^m \cap 
\Gamma((k,j),q)$ for every $(k,j).$ Consequently we can assume that $\{x_l,\ \ l=1,2,\dots\}\cap C(f)=\emptyset.$

Thus there is a
sequence $x^l\to \infty$, such that for every $I_i$ there are
integers $(k_i,j_i),$
such that $\| x^l\| M_{I_i}/M_{I_r(k_i,j_i)}(x_l)\to 0$ and $f(x^l)\to y$.
This also gives
$y\in \Gamma((k,j),q) \cap \C^m$ with $((k,j),q) = ((k_1,j_1),\dots,(k_s,j_s),q)$.

Conversely, if
$y\in \Gamma((k,j),q)\cap \C^m,$
then we can choose a sequence
$x^l\to \infty$, such that
$f(x^l)\to y$ and
$\|x^l\|M_{I_r}/M_{I_r(k_r,j_r)}(x^l)\to 0.$
It is easy to observe  that this implies
$\|x^l\|g'(x^l)\to 0$ and
$f(x^l)\to y$, i.e. $y\in K_\infty (f).$
\end{proof}
Now in light of \cite[Theorem 3.3]{Jelonek2005}, we have that $K_\infty(f)\not=\C^m$ hence $\C^m \cap \bigcup_{((k,j),q)} \Gamma((k,j),q)\not=\C^m.$ By Lemma \ref{FormulaKf}, $K_\infty(f)$ is an algebraic set. The theorem follows.
\end{proof}

\subsection{ A sketch of an algorithm}\label{subsection}

Let $X:=B_i\subset \C^{n+1}$ be a smooth affine variety of dimension $n-r$ and let
$I(X)=\{ b_1,\ldots, b_w\}.$
Let $f=(f_1,\dots, f_m): X \to \C^m$ be a polynomial dominant mapping.
Then the set $K_\infty(f)$ can be computed as follows.

By  construction $B_i$ is the subset of complete intersection, hence we 
 can choose polynomials $b_1,\dots,b_r\in I(X)$ such that
rank $\{ \grad \ b_1,\dots, \grad \ b_r\}=r$ on 
$X$.      Let us consider
the rational mapping:
$$\Phi((k_1,j_1),\dots,(k_s,j_s),q) :X \ni x \mapsto
(f(x), W_{I_1(k_1,j_1)}(x),x_1W_{I_1(k_1,j_1)}(x)
,\dots, x_nW_{I_1(k_1,j_1)}(x),$$
$$\dots, W_{I_s(k_s,j_s)}(x), x_1W_{I_s(k_s,j_s)}(x)
      ,\dots, x_n W_{I_s(k_s,j_s)}(x), 1/x_{q})
\in \C^m\times\C^{N},$$
which are constructed exactly as in the proof of
Theorem \ref{glowne}. Recall that
$$\Gamma((k_1,j_1),\dots,(k_s,j_s),q) =
cl(\Phi((k_1,j_1),\dots,(k_s,j_s),q) (X)).$$
We know also that
$$K_\infty (f)= L\cap \left( \bigcup_{((k_1,j_1),\dots,(k_s,j_s)),q}
\Gamma((k_1,j_1),\dots,(k_s,j_s),q)\right),$$
where $L=\C^m\times (0,\dots,0).$
First we  compute the ideal of the set $\Gamma((k_1,j_1),\dots,(k_s,j_s),q).$

To  this  end we   restrict the mapping $\Phi((k,j),q)$  to an open dense
subset $U$ on which this mapping  is regular.
In particular we can choose the set $U=X\setminus(\bigcup^s_{r=1}
\{M_{I_r(k_r,j_r)}=0\} \cup \{x_q=0\}).$
The set $U$ can be identified with the set $$V((k_1,j_1),\dots,(k_s,j_s),q)
:= $$
$$
:=
\{ (x,t,z_1,\dots,z_s)\in \C^{n+1}\times \C\times \C^s : b_j=0, \\j=1,\dots, w;\\ x_qt=1; \\
M_{I_r(k_r,j_r)}z_r=1; \ \
r=1,\dots,s\}.$$
Now we can  consider a morphism
\noindent
$$\Psi((k_1,j_1),\dots,(k_s,j_s)) :V((k_1,j_1),\dots,(k_s,j_s),q) \to
  \C^m\times\C^{N}.$$
defined by $$ (x,z) \to
(f(x), z_1M_{I_1}(x),x_1z_1M_{I_1}(x), \dots
, x_nz_1M_{I_1}(x),$$
$$\linebreak
\dots , z_sM_{I_s}(x), x_1z_sM_{I_s}(x)
      ,\dots, x_n z_s M_{I_s}(x),t).$$

\noindent Denote $\Psi((k_1,j_1),\dots,(k_s,j_s),q):=
(\psi_1(x,z),\dots, \psi_{m+N}(x,z)).$
It is easy to see that
$$\Gamma((k_1,j_1),\dots,(k_s,j_s),q)$$
is the closure of $$
\Psi((k_1,j_1),\dots,(k_s,j_s),q) (V((k_1,j_1),\dots,(k_s,j_s)),q).$$
Let $G((k_1,j_1),\dots,(k_s,j_s),q)=\graph(\Psi((k_1,j_1),\dots,(k_s,j_s),q))$.
A basis of the ideal $I$
of the set $G((k_1,j_1),\dots,(k_s,j_s),q)$ in the ring $\C[x_1,\dots,x_n,t,z_1,\dots,z_s;y_1,\dots,y_{m+N}]$ is given
by  the polynomials $$\{b_j, \\j=1,\dots, w;\} \cup \{z_rM_{I_r(k_r,j_r)}(x)-1,\\ r=1,\dots,s\}\cup \{tx_q-1\}\cup
\{ y_i-\psi_i(x,z), \\ i=1,\dots,m+N \}.$$
To compute a basis
${\mathcal B}((k_1,j_1),\dots,(k_s,j_s),q)$
of the ideal of the set $cl(\Gamma((k_1,j_1),\dots,(k_s,j_s),q),$ it
is enough to compute a Gr\"obner basis ${\mathcal
A}((k_1,j_1),\dots,(k_s,j_s))$ of the ideal $I$ in $\C[x,t,z,y]$ with respect to the lexicographic order in which $y<x,t,z$  (see e.g. \cite{Pauer1988}) and then to take
$${\mathcal B}((k_1,j_1),\dots,(k_s,j_s),q) =
{\mathcal A}((k_1,j_1),\dots,(k_s,j_s),q) \cap  \C [y_1,\dots,y_{m+N}].$$
Consequently,
$$K_\infty(f)= \bigcup_{((k_1,j_1),\dots,(k_s,j_s)),q}\{ y\in\C^m : \
h(y,0,\dots,0)=0\ \text{for  every} \  h\in
{\mathcal B}((k_1,j_1),\dots,(k_s,j_s),q) \}.$$

\vspace{3mm}
The computation of the set $K_0(f|_{B_i})$ is standard. 
Consider the set $$U:=
\{ x\in \C^{n+1} : b_j=0, \\j=1,\dots, w; \\
M_{I_r}=0; \ \
r=1,\dots,s\}.$$
Now we can  consider a morphism $f:U\to\C^m.$ We have $K_0(f)=\overline{f(U)}.$
Let $\Gamma$ be a graph of $f|_U$ and $I=I(\Gamma).$

A basis of the ideal $I$
 is given
by  the polynomials $$\{b_j; \\j=1,\dots, w;\} \cup \{M_{I_r}(x);\\ r=1,\dots,s\}\cup
\{ y_i-f_i; \\ i=1,\dots,m\}.$$
To compute a basis
${\mathcal B}$
of the ideal $I$ it
is enough to compute a Gr\"obner basis ${\mathcal
A}$ of the ideal
$I$
in $\C [x_1,\dots,x_{n+1};y_1,\dots,y_{m}]$ and then to take
$${\mathcal B} =
{\mathcal A} \cap  \C [y_1,\dots,y_{m}].$$
Consequently,
$K_0(f|_{B_i})= \bigcup \{ y\in\C^m : \ h(y,0,\dots,0)=0\ \text{for every} \  h\in {\mathcal B} \}.$

\vspace{5mm}

\noindent\textbf{Acknowledgment.} We would like to thank Nguyen Xuan Viet Nhan and Nguyen Hong Duc for some helpful discussions during the preparation the paper.


\begin{thebibliography}{99}

\bibitem{DAcunto2005}
D. D'Acunto, V. Grandjean,{\em On gradient at infinity of semialgebraic functions,} Ann. Pol. Math. 87 (2005), 39--49. 

\bibitem{grebner} D. Cox, J. Little, D. O'Shea, {\em Ideals, Varieties, and Algorithms,} Springer Verlag, New York, 2007.

\bibitem{Ehresmann1950}
C. Ehresmann, {\em Les connexions infinit\'esimales dans un espace fibr\'e diff\'erentiable,} Colloque de Topologie, Bruxelles (1950), 29--55.

\bibitem{Eisenbud1992}
D. Eisenbud, C. Huneke, W. Vasconcelos, {\em Wolmer direct methods for primary decomposition,} Invent. Math. 110 (1992), no. 2, 207--235.

\bibitem{Flores2016}
A. G. Flores and B. Teissier, {\em Local polar varieties in the geometric study of singularities,} (2016) (preprint).

\bibitem{radical} Gianni, P.; Trager, B.; Zacharias, G., {\em Gröbner Bases and Primary Decomposition 
of Polynomial Ideals}, J. Symb. Comp. 6, 149–167 (1988).



\bibitem{Greuel2008}
G-M. Greuel, G. Pfister, {\em A singular introduction to commutative algebra,} Second, extended edition. With contributions by Olaf Bachmann, Christoph Lossen and Hans Schönemann, Springer, Berlin, 2008. xx+689 pp.

\bibitem{Jelonek2005}
Z. Jelonek and K. Kurdyka, {\em Quantitative generalized Bertini--Sard theorem for smooth affine varieties}, Discrete Comput. Geom. 34, no. 4 (2005), 659--678.

\bibitem{Krick1991}
T. Krick and A. Logar, {\em An algorithm for the computation of the radical of an ideal in the ring of polynomials,} Applied algebra, algebraic algorithms and error-correcting codes (New Orleans, LA, 1991), 195--205, Lecture Notes in Comput. Sci., 539, Springer, Berlin, 1991. 

\bibitem{kos} K. Kurdyka, P. Orro, S. Simon, {\em  Semialgebraic Sard
theorem for generalized
critical values,} {\em Journal of Differential Geometry 56,} 67-92, 2000.

     




\bibitem{Mather2012}
J. Mather, {\em Notes on topological stability,} Bull. Amer. Math. Soc. (N.S.) 49 (2012), no. 4, 475--506.

\bibitem{Pauer1988}
F. Pauer and M. Pfeifhofer, {\em The theory of Gr\"obner basis}, L’Enseignement Math\'ematique 34 (1988), 215--232.

\bibitem{Rabier1997}
P. Rabier, {\em Ehresmann fibrations and Palais--Smale conditions for morphisms of Finsler manifolds,} Ann. of Math. (2) 146 (1997), no. 3, 647--691.


\bibitem{Teissier1982}
B. Teissier, {\em Vari\'et\'es polaires. II. Multiplicit\'es polaires, sections planes, et conditions de Whitney}, Algebraic geometry (La R\`abida, 1981), 314--491, Lecture Notes in Math., 961, Springer, Berlin, 1982.

\bibitem{Thom1969}
R. Thom, {\em Ensembles et morphismes stratifi\'es,} Bull. Amer. Math. Soc. 75 (1969), 240--284.

\bibitem{dimension} M. S. Uddin, {\em Computing dimension of affine varieties using Groebner basis approach}, IOSR Journal of Mathematics, vol. 8, issue 3, (2013), 36--39.

\bibitem{Verdier1976}
J-L. Verdier, {\em Stratifications de Whitney et th\'eor\`eme de Bertini-Sard,} Invent. Math. 36 (1976), 295--312.

\bibitem{Whitney1965-1}
H. Whitney, {\em Local properties of analytic varieties,} Differential and Combinatoric Topology, Princeton University Press, Princeton, (1965), 205--244.

\bibitem{Whitney1965-2}
H. Whitney, {\em Tangents to an analytic variety,} Ann. of Math. 81 (1965), 496--549.



\end{thebibliography}
\end{document}